\documentclass[11pt]{amsart}

\pdfoutput=1
\usepackage{geometry} 
\usepackage{graphicx,url}
\usepackage{amssymb,amsmath, amsthm}
\usepackage{epstopdf}
\usepackage{subfig}
\usepackage[applemac]{inputenc} 

\newtheorem{theorem}{Theorem}
\newtheorem{remark}{Remark}[section]
\newtheorem{corollary}[theorem]{Corollary}

\title{A family of random walks with generalized Dirichlet steps}
\numberwithin{equation}{section}
\author{ Alessandro De Gregorio}

\address{Dipartimento di Scienze Statistiche,
``Sapienza" University of Rome,
P.le Aldo Moro, 5 - 00185, Rome, Italy}
\email{alessandro.degregorio@uniroma1.it}
\allowdisplaybreaks
\begin{document}

\maketitle
\begin{abstract}
We analyze a class of continuous time random walks in $\mathbb R^d,d\geq 2,$ with uniformly distributed directions. The steps performed by these processes are distributed according to a generalized Dirichlet law. Given the number of changes of orientation, we provide the analytic form of the probability density function of the position $\{\underline{\bf X}_d(t),t>0\}$ reached, at time $t>0$, by the random motion. In particular, we analyze the case of random walks with two steps. 

In general, it is an hard task to obtain the explicit probability distributions for the process $\{\underline{\bf X}_d(t),t>0\}$	. Nevertheless, for suitable values for the basic parameters of the generalized Dirichlet probability distribution, we are able to derive the explicit conditional density functions of $\{\underline{\bf X}_d(t),t>0\}$. Furthermore, in some cases, by exploiting the fractional Poisson process, the unconditional probability distributions are obtained. 
This paper extends in a more general setting, the random walks with Dirichlet displacements introduced in some previous papers. 
\end{abstract}

\section{Introduction}
Several authors over the years analyzed continuous time non-Markovian random walks, that describe the motion of a non-interacting particle walking in straight lines and turning through any angle whatever, i.e. with uniformly distributed directions on the unit sphere. Usually, the main object of interest is represented by the position reached by the particle after a fixed or random number of steps. These stochastic processes are called ``Pearson random walks" or equivalently ``random flights". Many statistical physics problems can be studied by taking into account these random motions. We discuss two examples (see Hughes, 1995, and Weiss, 2002, for other applications):
\begin{itemize}
\item the Pearson walks are equivalent to the problem of addition, at the time $t>0$,  of waves with the same frequency $\omega$ e and arbitrary phase $\phi$, that is $X(t)=\sum_{k=1}^nX_k(t)=\sum_{k=1}^na_k\cos(\omega t+\phi_j),$ with $a_k>0$. For instance, an application of the random flights arises in the analysis of the ``laser speckle''. This phenomenon is represented by the speckles caused by the superposition of laser light waves reflected by a surface. 
The scattered waves interfere in different ways on the surface and then at some points the interference leads to strong brightness, while at other points it produces low intensity. Then the laser speckle can be modeled by means of the Pearson walk $X(t)$;
\item the random flights provide a model of the statistical mechanics of a polymeric chains. Indeed, a backbone of a linear polymer consists in a sequence of atoms. It may be displayed as a sequence of straight lines connecting the centres of the atoms on the backbone. Several configurations over the time are possible. Therefore, we can describe the configurations in terms of random flights.
\end{itemize} 

Many papers (see, for instance, Stadje, 1987, Masoliver {\it et al}. 1993, Orsingher and De Gregorio, 2007 and Garcia-Pelayo, 2012) analyzed random flights in the multidimensional real spaces. Usually the main assumption concerns the underlying homogeneous Poisson process governing the changes of direction of the random walk. Poisson paced times imply that the time lapses are exponentially distributed. Unfortunately, under these assumptions, the explicit distribution of the random flight is obtained only in two spaces: $\mathbb R^2$ and $\mathbb R^4$. 

In the last years, the research in this field is mainly devoted to the analysis of random motions with non-uniformly distributed lengths of the time displacements between consecutive changes of direction.
Indeed, the exponentially distributed steps assign high probability mass to short intervals and for this reason they are not suitable for many important applications in physics, biology, and engineering. For example, Beghin and Orsingher (2010) introduced a random motion which changes direction at even-valued Poisson events; this implies that the time between successive deviations is a Gamma random variable. This model can also be interpreted as the motion of particles that can hazardously collide with obstacles of different size, some of which are capable of deviating the motion. Recently, multidimensional random walks with Gamma intertimes have been also taken into account by Le Ca\"er (2011) and Pogorui and Rodriguez-Dagnino (2011), (2013). Le Ca\"er (2010) and De Gregorio and Orsingher (2012) considered the joint distribution of the time displacements as Dirichlet random variables with parameters depending on the space in which the random walker performs its motion. The Dirichlet law assigns higher probability to time displacement with intermediate length and allows to explicit (for suitable choices of the parameters), for each space $\mathbb{R}^d,d\geq 2,$ the exact probability distribution of the position reached by the random motion at time $t>0$. De Gregorio (2012) dealt with a random flight in $\mathbb R^d$ with Dirichlet steps and non-uniformly distributed directions, while Pogorui and Rodriguez-Dagnino (2012) considered random flights with uniformly distributed directions and random velocity.

The aim of this paper is to analyze the multidimensional random walks with generalized Dirichlet displacements. The generalized Dirichlet distribution, introduced by Connor and Mosimann (1969), has a more general covariance structure  than classical Dirichlet distribution and, furthermore, it contains a greater number of parameters. For these reasons it allows to capture many features of the analyzed phenomenon.

The paper is organized as follows. In Section \ref{sec1}, we introduce the random walks studied in this paper, we indicate its position at time $t>0$ by $\{\underline{\bf X}_d(t),t>0\}$. Given the number of the steps, $n\in\mathbb N$, performed by the random motion, in Section \ref{sec2} we provide the expression of its density function in integral form.  For $n=1$, some explicit results are obtained. Section \ref{sec3} gives the exact probability distributions of $\{\underline{\bf X}_d(t),t>0\}$, when the parameters of the generalized Dirichlet distribution are suitably chosen. We obtain several distributions and the related random walks are compared. 
Section \ref{sec4} contains some remarks on the unconditional probability distributions of $\{\underline{\bf X}_d(t),t>0\}$.

\section{Description of the random walks}\label{sec1}

We describe the class of random walks studied in this paper. Let us consider a particle or a walker which starts from the origin of $\mathbb{R}^d,d\geq 2$, and performs its motion with a constant velocity $c>0$. We indicate by $0=t_0<t_1<t_2<...<t_k<...$ the random instants at which the random walker changes direction and denote the length of time separating these instants by $\tau_k=t_k-t_{k-1}$, $k\geq 1$. Let $\mathcal N(t)=\sup\{k\geq 1:t_k\leq t\}$ be the (random) number of times in which the random motion changes direction during the interval $[0,t]$. If, at time $t>0$, one has that $\mathcal N(t)=n$, with $n\geq 1$, the random motion has performed $n+1$ steps. We observe that ${\bf \underline\tau}_n=(\tau_1,...,\tau_n)\in S_n$, where $S_n$ represents the open simplex
$$ S_n=\left\{(\tau_1,...,\tau_n)\in\mathbb R^n: 0<\tau_k<t-\sum_{j=0}^{k-1}\tau_j, k=1,2,...,n\right\},$$ with $\tau_0=0$ and $\tau_{n+1}=t-\sum_{j=1}^n\tau_j$.

The directions of the particle are represented by the points on the surface of the $(d-1)$-dimensional sphere with radius one.  We denote by $\underline{\theta}_{d-1}=(\theta_1,...,\theta_{d-2},\phi)$ the random vector  (independent from ${\bf \underline\tau}_n$) representing the orientation of the particle; we assume that $\underline{\theta}_{d-1}$ has uniform distribution on the $(d-1)$-sphere with radius one $\mathbb{S}_{1 }^{d-1}=\{\underline{\bf x}_d\in \mathbb R^d:||\underline{\bf x}_d||=1\}$. Therefore, by observing that meas$(\mathbb{S}_{1 }^{d-1})=\frac{2\pi^{\frac d2}}{\Gamma(\frac d2)}$ and the volume element of $\mathbb{S}_{1 }^{d-1}$ is given by $$d\mathbb{S}_{1 }^{d-1}=\sin^{d-2}\theta_1\sin^{d-3}\theta_2\cdots \sin\theta_{d-2} d\theta_1d\theta_2\cdots d\theta_{d-2},$$ the probability density function of $\underline{\theta}_{d-1}$ is equal to
\begin{equation}\label{eq:jointdis1}
\varphi(\underline{\theta}_{d-1})=\frac{\Gamma(\frac d2)}{2\pi^{\frac d2}}\sin^{d-2}\theta_1\sin^{d-3}\theta_2\cdots \sin\theta_{d-2},
\end{equation}
where $\theta_k\in[0,\pi],k\in\{1,...,d-2\},\,\phi\in[0,2\pi]$. Furthermore the particle chooses a new direction independently from the previous one.

Let us denote with  $\{\underline{\bf X}_d(t),t>0\}$ the process representing the position reached, at time $t>0$, by the particle moving randomly according to the rules described above. The position $\underline{\bf X}_d(t)=(X_1(t),...,X_d(t)),$ is the main object of investigation. By setting $\mathcal N(t)=n$, the random walk $\{\underline{\bf X}_d(t),t>0\},$ can be described in the following manner
\begin{equation}\label{eq:definition}
\underline{\bf X}_d(t)=c\sum_{k=1}^{n+1}{\bf \underline{V}}_k\tau_k,
\end{equation}
where ${\bf\underline{V}}_k,k=1,2,...,n+1,$ are independent $d$-dimensional random vectors defined as follows
$${\bf\underline{V}}_k=
\left(
\begin{array}{l}
V_{1,k}\\
V_{2,k}\\
...\\
V_{d-1,k}\\
V_{d,k}
	\end{array}\right)=\left(
 \begin{array}{l} 
 \cos\theta_{1,k}\\
    \sin\theta_{1,k}\cos\theta_{2,k}\\
     
     ...\\
     \sin\theta_{1,k}\sin\theta_{2,k}\cdot\cdot\cdot\sin\theta_{d-2,k}\cos\phi_{k} \\
           \sin\theta_{1,k}\sin\theta_{2,k}\cdot\cdot\cdot\sin\theta_{d-2,k}\sin\phi_{k} 
  
   \end{array}    
    \right)$$
    and $(\theta_{1,k},\theta_{2,k},...,\theta_{d-2,k},\phi_{k})$ has distribution \eqref{eq:jointdis1}.

A crucial role is played by the random vector ${\bf \underline\tau}_n$. We assume that the intervals $\tau_ks$ have the following joint density function
\begin{equation}\label{eq:gdd}
f({\bf \underline\tau}_n; \mathbf{\underline{a}}_{n},\mathbf{\underline{b}}_{n},t)=C(\mathbf{\underline{a}}_{n},\mathbf{\underline{b}}_{n},t)\prod_{k=1}^{n}\left[\tau_k^{a_k-1}(t-\sum_{j=1}^k\tau_j)^{b_k-1}\right], \quad {\bf \underline\tau}_n\in S_n,
\end{equation}
where 
$$C(\mathbf{\underline{a}}_{n},\mathbf{\underline{b}}_{n},t)=\frac{1}{t^{\sum_{k=1}^n(a_k+b_k-1)}}\prod_{k=1}^n\frac{\Gamma(1+\sum_{j=k}^n(a_j+b_j-1))}{\Gamma(a_k)\Gamma(b_k+\sum_{j=k+1}^n(a_j+b_j-1))},$$
with $\mathbf{\underline{a}}_{n}=(a_1,a_2,...,a_n), \mathbf{\underline{b}}_{n}=(b_1,b_2,...,b_n)$ and $a_1,...,a_n,b_1,...,b_n>0$. The density function $f({\bf \underline\tau}_d; \mathbf{\underline{a}}_{n},\mathbf{\underline{b}}_{n},t)$ represents a rescaled generalized Dirichlet distribution (see (2.1) in Chang {\it et al}., 2010, and references therein) that we indicate by $GD(\mathbf{\underline{a}}_{n};\mathbf{\underline{b}}_{n})$. 
This assumption leads to a family of random motions which contains, as particular cases, the processes studied in  Le Ca\"er (2010) and De Gregorio and Orsingher (2012) in the classical Dirichlet setting. By setting $\mathbf{\underline{b}}_{n}=(1,1,...,1,b_n)$, \eqref{eq:gdd} becomes
\begin{equation*}
\frac{\Gamma\left(\sum_{k=1}^{n}a_k+b_n\right)}{\prod_{k=1}^{n}\Gamma\left(a_k\right)\Gamma(b_n)}\frac{1}{t^{\sum_{k=1}^{n}a_k+b_n-1}}\prod_{k=1}^{n}\tau_k^{a_k-1}(t-\sum_{k=1}^n\tau_k)^{b_n-1},
\end{equation*}
which, for $t=1$, is the well-known Dirichelt distribution. If $\mathbf{\underline{a}}_{n}=\mathbf{\underline{b}}_{n}=(1,...,1),$ $GD(\mathbf{\underline{a}}_{n};\mathbf{\underline{b}}_{n})$ becomes the uniform distribution $n!/t^n$ in the simplex $S_n$, appearing in the case of Poisson intertimes. 

In general, it is known that the explicit probability distribution of the Pearson walk $\underline{\bf X}_d(t)$ exists in few cases. The main gain of the above assumption on the random vector ${\bf \underline\tau}_n$,  is that, by using a suitable parametrization of \eqref{eq:gdd}, we get random walks with explicit density functions (see Section \ref{sec3}).

It is worth to point out that the stochastic process treated here can be represented by means of the triple $(\underline{\theta}_{d-1},\underline{\tau}_n,\mathcal{N}(t))$ of independent random vectors: $\underline{\theta}_{d-1}=(\theta_1,...,\theta_{d-2},\phi)$ is the orientation of displacements, with uniform law \eqref{eq:jointdis1}, $\underline{\tau}_{n}=(\tau_1,...,\tau_n)$ represents the time displacements, with distribution $GD(\mathbf{\underline{a}}_{n};\mathbf{\underline{b}}_{n})$, and $\mathcal{N}(t)$ is the number of changes of direction. In the models analyzed by Orsingher and De Gregorio (2007), $\underline{\theta}_{d-1}$ has law coinciding with \eqref{eq:jointdis1},  $\underline{\tau}_{n}$ is uniformly distributed and $\mathcal{N}(t)$ is a homogeneous Poisson process. 
To conclude this section, we remark that the sample paths of \eqref{eq:definition} appear like joined straight lines representing the randomly oriented steps with random lengths.

\section{General results}\label{sec2}
In this section we assume that the number of steps performed by the random walk is fixed and equal to $n+1$, with $n\geq 1$ (i.e. $\mathcal N(t)=n$). Here we  consider the general case in which the random vector $\underline \tau_n$, representing the random displacements between consecutive deviations, follows the generalized Dirichlet distribution \eqref{eq:gdd}. 

\subsection{The general case}
Let $\underline{{\bf x}}_d=(x_1,...,x_d)$, $d\underline{{\bf x}}_d=(dx_1,...,dx_d)$ and $\underline{\alpha}_d=(\alpha_1,...,\alpha_d)\in \mathbb{R}^d$.  Let us indicate by $||\underline{{\bf x}}_d||=\sqrt{\sum_{k=1}^dx_k^2}$ and $<\underline{\alpha}_d,\underline{{\bf x}}_d>=\sum_{k=1}^d\alpha_kx_k$ the Euclidean distance and the scalar product, respectively. The conditional characteristic function of $\{\underline{{\bf X}}_d(t),t>0\}$ is denoted by $$\mathcal F_n(\underline{\bf \alpha}_d)={\bf E} \left.\left\{e^{i<\underline{\alpha}_d,\underline{\bf X}_d(t)>}\right|\mathcal{N}(t)=n\right\}.$$
We remark that at time $t>0$, the random flight  $\{\underline{{\bf X}}_d(t),t>0\}$ with at least two steps, is located inside the ball $\mathcal{B}_{ct}^d=\{\underline{{\bf x}}_d\in\mathbb{R}^d:||\underline{{\bf x}}_d||<ct\}$ with center the origin of $\mathbb{R}^d$ and radius $ct$.

 Let $P(\cdot\in d\underline{{\bf x}}_d|\mathcal{N}(t)=n)=P_n(\cdot\in d\underline{{\bf x}}_d)$.
The first result concerns the conditional density function of $\{\underline{{\bf X}}_d(t),t>0\}$.

\begin{theorem}\label{th1}
Given $\mathcal N(t)=n$, $n\geq 1,$
the density function of $\{\underline{{\bf X}}_d(t),t>0\}$ is equal to

\begin{align}\label{eq:density}
p_n(\underline{{\bf x}}_d,t)&=\frac{P(\underline{{\bf X}}_d(t)\in d\underline{{\bf x}}_d|\mathcal{N}(t)=n)}{\prod_{i=1}^ddx_i}\notag\\
&=\frac{\left\{2^{\frac d2-1}\Gamma\left(\frac d2\right)\right\}^{n+1}}{(2\pi)^{\frac d2}||\underline{\bf x}_d||^{\frac d2-1}}\notag \\
&\quad\cdot\int_0^{\infty} \rho^{\frac d2}J_{\frac d2-1}(\rho ||\underline{\bf x}_d||)
d\rho\int_{S_n}f({\bf \underline\tau}_n; \mathbf{\underline{a}}_{n},\mathbf{\underline{b}}_{n},t)\prod_{k=1}^{n+1}\left\{\frac{J_{\frac d2-1}(c\tau_k\rho)}{(c\tau_k\rho)^{\frac d2-1}} \right\}\prod_{k=1}^{n}d\tau_k,
\end{align}
where $\underline{{\bf x}}_d\in\mathcal B_{ct}^d$ and
$$J_\nu(x)=\sum_{k=0}^\infty(-1)^k\left(\frac{x}{2}\right)^{2k+\nu}\frac{1}{k!\Gamma(k+\nu+1)},\quad x,\nu\in\mathbb{R},$$
is the Bessel function.
\end{theorem}

\begin{proof} Let us start the proof showing that under the assumption \eqref{eq:gdd}, the characteristic function of $\underline{\bf X}_d(t)$ is equal to 
\begin{align}\label{eq:cf}
\mathcal F_n(\underline{\bf \alpha}_d)
&=\left\{2^{\frac d2-1}\Gamma\left(\frac d2\right)\right\}^{n+1}\int_{S_n}f({\bf \underline\tau}_d; \mathbf{\underline{a}}_{n},\mathbf{\underline{b}}_{n},t)\prod_{k=1}^{n+1}\left\{\frac{J_{\frac d2-1}(c\tau_k||\underline{\alpha}_d||)}{(c\tau_k||\underline{\alpha}_d||)^{\frac d2-1}} \right\}\prod_{k=1}^{n}d\tau_k.
\end{align}
 We can write that
\begin{align*}
\mathcal F_n(\underline{\bf \alpha}_d)&=\int_{S_n}f({\bf \underline\tau}_d; \mathbf{\underline{a}}_{n},\mathbf{\underline{b}}_{n},t)\,\mathcal{I}_n(\underline{\alpha}_d;{\bf \underline\tau}_d)\prod_{k=1}^nd\tau_k
\end{align*}
where
\begin{align}\label{eq:I_n}
\mathcal{I}_n(\underline{\alpha}_d;{\bf \underline\tau}_d)&=\int_0^\pi d \theta_{1,1}\cdots\int_0^\pi d \theta_{1,n+1}\cdots \int_0^\pi d \theta_{d-2,1}\cdots\int_0^\pi d \theta_{d-2,n+1} \int_0^{2\pi}d \phi_{1}\cdots\int_0^{2\pi} d \phi_{n+1}\notag\\
&\quad\cdot \prod_{k=1}^{n+1}\Bigg\{\exp\left\{ic\tau_k<\underline{\alpha}_d,{\bf \underline{V}}_k> \right\}\frac{\Gamma(d/2)}{2\pi^{d/2}}\sin\theta_{1,k}^{d-2}\cdot\cdot\cdot\sin\theta_{d-2,k}\Bigg\}.
\end{align}
It is known that the integral $\mathcal{I}_n(\underline{\alpha}_d;{\bf \underline\tau}_d)$ (see formula (2.5) in De Gregorio and Orsingher, 2012) is equal to
\begin{equation}\label{intangle}
\mathcal{I}_n(\underline{\alpha}_d;{\bf \underline\tau}_d)=\left\{2^{\frac d2-1}\Gamma\left(\frac d2\right)\right\}^{n+1}\prod_{k=1}^{n+1}\frac{J_{\frac d2-1}(c\tau_k||\underline{\alpha}_d||)}{(c\tau_k||\underline{\alpha}_d||)^{\frac d2-1}}
\end{equation}
and this leads to \eqref{eq:cf}.

Now, by inverting the characteristic function \eqref{eq:cf}, we are able to show that the 
density of the process $\{\underline{\bf X}_d(t),t>0\},$ is given by \eqref{eq:density}. Let us denote by 

$$\Theta=\left\{(\theta_1,...,\theta_{d-2},\phi)\in\mathbb{R}^{d-1}: \theta_i\in [0,\pi],\phi\in[0,2\pi],\, i=1,...,d-2\right\}$$
and by $\underline{v}_d$ the vector
$$\underline{v}_d=\left(
 \begin{array}{l} 
  \cos\theta_{1}\\
      \sin\theta_{1}\cos\theta_{2}\\
      ...\\
       \sin\theta_{1}\sin\theta_{2}\cdot\cdot\cdot\sin\theta_{d-2}\cos\phi \\
      \sin\theta_{1}\sin\theta_{2}\cdot\cdot\cdot\sin\theta_{d-2}\sin\phi 
   \end{array}    
    \right).$$
Therefore, by inverting the characteristic function \eqref{eq:cf} and by passing to the (hyper)spherical coordinates, we have, for $\underline{{\bf x}}_d\in\mathcal B_{ct}^d$, that
\begin{align*}
p_n(\underline{\bf x}_d,t)
&=\frac{1}{(2\pi)^d}\int_{\mathbb{R}^d}e^{-i<\underline{\alpha}_d,\underline{\bf x}_d>}\mathcal F_n(\underline{\bf \alpha}_d)d\alpha_1\cdots d\alpha_d\notag\\
&=\frac{1}{(2\pi)^d}\int_0^\infty \rho^{d-1}d\rho\int_{\Theta} e^{-i\rho<\underline{v}_d,\underline{\bf x}_d>}d\mathbb {S}_1^{d-1}\notag\\
&\quad\cdot \left\{2^{\frac d2-1}\Gamma\left(\frac d2\right)\right\}^{n+1}\int_{S_n}f({\bf \underline\tau}_n; \mathbf{\underline{a}}_{n},\mathbf{\underline{b}}_{n},t)\prod_{k=1}^{n+1}\left\{\frac{J_{\frac d2-1}(c\tau_k\rho)}{(c\tau_k\rho)^{\frac d2-1}} \right\}\prod_{k=1}^{n}d\tau_k.
\end{align*}
By means of formula (2.12) in De Gregorio and Orsingher (2012) 
\begin{equation}\label{eq: (2.12) DGO12}
\int_{\Theta} e^{-i\rho<\underline{v}_d,\underline{\bf x}_d>}d\mathbb {S}_1^{d-1}=(2\pi)^{\frac d2}\frac{J_{\frac d2-1}(\rho ||\underline{\bf x}_d||)}{(\rho ||\underline{\bf x}_d||)^{\frac d2-1}},
\end{equation}
and this concludes the proof.
\end{proof}

\begin{remark}
Actually, it is possible to obtain the density function of \eqref{eq:definition} also if $\underline{\tau}_n$ possesses an arbitrary density $g(\underline{\tau}_n;t)$ on the simplex $S_n$. In this case, by means of the same arguments used in the proof of Theorem \ref{th1}, we can check that the density function of \eqref{eq:definition} is given by the expression \eqref{eq:density} where in place of $f({\bf \underline\tau}_n; \mathbf{\underline{a}}_{n},\mathbf{\underline{b}}_{n},t)$ appears $g(\underline{\tau}_n;t)$.
\end{remark}

\begin{remark}\label{remis}The random process $\{\underline{\bf X}_d(t),t>0\}$ represents an isotropic random motion. The density function $p_n(\underline{{\bf x}}_d,t)$ is rotationally invariant and it depends on the distance $||\underline{{\bf x}}_d||$. Then we can write $p_n(\underline{{\bf x}}_d,t)=p_n(||\underline{{\bf x}}_d||,t)$. Furthermore, as consequence of the isotropy, the distribution of the radial process $\{R_d(t),t>0\}$, where $R_d(t)=||\underline{\bf X}_d(t)||$,  becomes
\begin{equation}
r^{d-1}p_n(r,t)\text{meas}(\mathbb{S}_{1}^{d-1}),\quad 0<r<ct.
\end{equation}
\end{remark}

\subsection{Random walks with two steps}
In general it is not possible to calculate explicetely the density function \eqref{eq:density}. Nevertheless some results can be obtained by setting $n=1$. In other words we consider a random motion $\{\underline{\bf X}_d(t),t>0\}$, which, at time $t>0$, has performed one deviation (or two displacements). In this case the generalized Dirichlet distribution \eqref{eq:gdd} becomes a Beta distribution with parameters $a_1$ and $b_1$. The next result represents a generalization of Theorem 2.3 in Orsingher and De Gregorio (2007).

\begin{theorem}
For $n=1$, we have that
\begin{align}\label{eq:densityonestep}
p_1(\underline{\bf x}_d,t)
&=\frac{1}{2^{d-2}\pi^{\frac{d+1}{2}}} \frac{\Gamma\left(\frac d2\right)^2}{\Gamma(\frac d2-\frac 12)}\frac{\Gamma\left(
a_1+b_1\right)}{\Gamma\left(a_1\right)\Gamma\left(b_1\right)}\frac{1}{t^{a_1+b_1-1}}\frac{(c^2t^2-\ ||\underline{\bf x}_d||^2)^{\frac{d-3}{2}}
}{||\underline{\bf x}_d||^{d-2}c^{2d-4}}\notag\\
&\quad\cdot\int_{\frac{t}{2}-\frac{ ||\underline{\bf x}_d||}{2c}}^{\frac{t}{2}+\frac{ ||\underline{\bf x}_d||}{2c}}\tau_1^{a_1-d+1}(t-\tau_1)^{b_1-d+1}[4c^2\tau_1(t-\tau_1)-c^2t^2+||\underline{\bf x}_d||^2]^{\frac{d-3}{2}}
d\tau_1
\end{align}
with $\underline{\bf x}_d\in\mathcal{B}_{ct}^d$.
\end{theorem}
\begin{proof} The result \eqref{eq:densityonestep} is derived by using the same arguments of the proof of Theorem 2.3 and formula (2.26) in Orsingher and De Gregorio (2007). Then we omit the calculations.

\end{proof}

\begin{corollary}
 If $a_1=b_1=a$, we obtain that
\begin{align*}
p_1(\underline{\bf x}_d,t)
=\frac{1}{2^{2a-d+1}\pi^{\frac{d}{2}}}\frac{\Gamma\left(\frac d2\right)\Gamma\left(
2a\right)}{\Gamma\left(a\right)^2}\frac{(c^2t^2-\ ||\underline{\bf x}_d||^2)^{\frac{d-3}{2}}
}{(ct)^{2d-3}}{}_2F_1\left(d-1-a,\frac12,\frac d2;\frac{||\underline{\bf x}_d||^2}{c^2t^2}\right),\end{align*}
where $\underline{\bf x}_d\in \mathcal{B}_{ct}^d,$ and
the hypergeometric function ${}_2F_1(\alpha,\beta,\gamma;z)$ is defined as follows
$$
{}_2F_1(\alpha,\beta,\gamma;z)=\frac{\Gamma(\gamma)}{\Gamma(\beta)\Gamma(\gamma-\beta)}\int_0^1t^{\beta-1}
(1-t)^{\gamma-\beta-1}(1-tz)^{-\alpha} dt,
$$
for $Re \,\gamma >Re\, \beta >0,|z|<1.$ 

\end{corollary}

\begin{proof}

We can write \eqref{eq:densityonestep} in the following alternative form
\begin{align}\label{eq:altform}
p_1(\underline{\bf x}_d,t)
&=\frac{1}{2^{d-2}\pi^{\frac{d+1}{2}}} \frac{\Gamma\left(\frac d2\right)^2}{\Gamma(\frac d2-\frac 12)}\frac{\Gamma\left(
a_1+b_1\right)}{\Gamma\left(a_1\right)\Gamma\left(b_1\right)}\frac{1}{t^{a_1+b_1-1}}\frac{(c^2t^2-\ ||\underline{\bf x}_d||^2)^{\frac{d-3}{2}}
}{c^{2d-4}}\\
&\quad\cdot\frac{1}{(2c)^{a_1+b_1-2d+3}}\int_{-1}^{1}(ct+||\underline{\bf x}_d||z)^{a_1-d+1}(ct-||\underline{\bf x}_d||z)^{b_1-d+1}(1-z^2)^{\frac{d-3}{2}}\notag
dz
\end{align}
In the last step above, we applied the successive substitutions
$\tau_1-\frac{t}{2}=y$ and $2cy=z||\underline{\bf x}_d||.$ 
From \eqref{eq:altform}, if $a_1=b_1=a$, by means of the position $z^2=w$, we obtain that
\begin{align}\label{eq:altform2}
p_1(\underline{\bf x}_d,t)
&=\frac{1}{2^{2a-d+1}\pi^{\frac{d+1}{2}}} \frac{\Gamma\left(\frac d2\right)^2}{\Gamma(\frac d2-\frac 12)}\frac{\Gamma\left(
2a\right)}{\Gamma\left(a\right)^2}\frac{(c^2t^2-\ ||\underline{\bf x}_d||^2)^{\frac{d-3}{2}}
}{(ct)^{2d-3}}\\
&\quad\cdot\int_{0}^{1}w^{-\frac12}\left(1-\frac{||\underline{\bf x}_d||^2}{c^2t^2}w\right)^{a-d+1}(1-w)^{\frac{d-3}{2}}
dz\notag\\
&=\frac{1}{2^{2a-d+1}\pi^{\frac{d}{2}}}\frac{\Gamma\left(\frac d2\right)\Gamma\left(
2a\right)}{\Gamma\left(a\right)^2}\frac{(c^2t^2-\ ||\underline{\bf x}_d||^2)^{\frac{d-3}{2}}
}{(ct)^{2d-3}}{}_2F_1\left(d-1-a,\frac12,\frac d2;\frac{||\underline{\bf x}_d||^2}{c^2t^2}\right)\notag
\end{align}
with $ \underline{\bf x}_d\in \mathcal{B}_{ct}^d$.
\end{proof}

\begin{remark}
From \eqref{eq:altform2}, for $a=d-1$ we derive the following simplified density function
\begin{equation*}
p_1(\underline{\bf x}_d,t)=\frac{1}{2^{d-1}\pi^{\frac{d+1}{2}}} \frac{\Gamma\left(\frac d2\right)\Gamma(2(d-1))}{(\Gamma(d-1))^2}\frac{(c^2t^2-\ ||\underline{\bf x}_d||^2)^{\frac{d-3}{2}}}{(ct)^{2d-3}},
\end{equation*}
while for $a=d$ we obtain that
\begin{equation*}
p_1(\underline{\bf x}_d,t)=\frac{1}{2^{d+1}\pi^{\frac{d+1}{2}}} \frac{\Gamma\left(\frac d2\right)\Gamma(2d)}{(\Gamma(d))^2}\frac{(c^2t^2- ||\underline{\bf x}_d||^2)^{\frac{d-3}{2}}(c^2t^2-\frac1d||\underline{\bf x}_d||^2)}{(ct)^{2d-1}}.
\end{equation*}
Furthermore, since ${}_2F_1\left(\frac12,\frac12,\frac 32;z^2\right)=\arcsin z/z$, (see Lebedev, 1972, pag.259, formula (9.8.5)) for $d=3$ and $a=3/2$, we have that
$$p_1(\underline{\bf x}_3,t)=\frac{2}{\pi^2(ct)^2}\frac{\arcsin\left(\frac{||\underline{\bf x}_3||}{ct}\right)}{||\underline{\bf x}_3||}.$$
\end{remark}

It is worth to point out that $\mathbb R^3$ represents a suitable environment for analyzing the random flight (at least for $n=1$). Indeed, for $d=3$ the formula \eqref{eq:densityonestep} becomes 
\begin{align}\label{eq:densityd3n1}
p_1(\underline{\bf x}_3,t)
=\frac{1}{2^3\pi} \frac{1}{c^2||\underline{\bf x}_3||}\frac{\Gamma\left(
a_1+b_1\right)}{\Gamma\left(a_1\right)\Gamma\left(b_1\right)}\frac{1}{t^{a_1+b_1-1}}\int_{\frac{t}{2}-\frac{ ||\underline{\bf x}_3||}{2c}}^{\frac{t}{2}+\frac{ ||\underline{\bf x}_3||}{2c}}\tau_1^{a_1-2}(t-\tau_1)^{b_1-2}
d\tau_1,
\end{align}
where $\underline{\bf x}_3\in\mathcal{B}_{ct}^3$.
In some cases, we are able to obtain the explicit form for the density function \eqref{eq:densityd3n1}. For $a_1\neq 1$ and $b_1=2$, we get
\begin{align*}
p_1(\underline{\bf x}_3,t)
&=\frac{1}{2\pi(2ct)^{a_1+1}||\underline{\bf x}_3||}\frac{a_1(a_1+1)}{a_1-1}[(ct+||\underline{\bf x}_3||)^{a_1-1} -(ct-||\underline{\bf x}_3||)^{a_1-1}]
\end{align*}
and if $a_1=2$ the above result allows to the uniform distribution inside the three-dimensional ball with radius $ct$, that is $$p_1(\underline{\bf x}_3,t)=\frac{1}{2^2\pi}\frac{3}{(ct)^3}.$$ By setting $a_1=1$ and $b_1=2$ in \eqref{eq:densityd3n1}, we derive that
\begin{align*}
p_1(\underline{\bf x}_3,t)
&=\frac{1}{\pi(2ct)^2 ||\underline{\bf x}_3||}\log\left(\frac{ct + ||\underline{\bf x}_3||}{ct - ||\underline{\bf x}_3||}\right),
\end{align*}
which coincides with the distribution obtained in Orsingher and De Gregorio (2007) (see (2.27b) in Theorem 2.4) for uniformly distributed intertimes. 

We observe that \eqref{eq:densityd3n1}, can be also expressed as follows
\begin{align}\label{eq:densbeta}
p_1(\underline{\bf x}_3,t)&=\frac{1}{2^3\pi} \frac{1}{c^2t^2||\underline{\bf x}_3||}\frac{1}{B(a_1,b_1)}\notag\\
&\quad\cdot \left[ B\left(\frac{1}{2}+\frac{ ||\underline{\bf x}_3||}{2ct};a_1-1,b_1-1\right)-B\left(\frac{1}{2}-\frac{ ||\underline{\bf x}_3||}{2ct};a_1-1,b_1-1\right)\right]
\end{align}
where $B(x;a,b)=\int_0^xz^{a-1}(1-z)^{b-1}dz$ represents the incomplete Beta function, whereas $B(a,b)$ is the standard Beta function, for $a,b>0$. It is well-known that the incomplete beta function can be expressed in terms of distribution function of a binomial random variable. In other words, the following relationship holds
\begin{equation}\label{eq:relincombeta}
\frac{B(x;a,b)}{B(a,b)}=\sum_{k=a}^{a+b-1}\binom{a+b-1}{k}x^k(1-x)^{a+b-1-k},\quad a,b\in\mathbb N.
\end{equation}
If $a_1-1,b_1-1\in \mathbb N$, we have that \eqref{eq:relincombeta} turns out 
\begin{align*}
p_1(\underline{\bf x}_3,t)&=\frac{1}{2^3\pi} \frac{1}{c^2t^2||\underline{\bf x}_3||}\frac{(a_1+b_1-1)(a_1+b_1-2)}{(a_1-1)(b_1-1)}\sum_{k=a_1-1}^{a_1+b_1-3}\binom{a_1+b_1-3}{k}\notag\\
&\quad\cdot\left[ \left(\frac{1}{2}+\frac{ ||\underline{\bf x}_3||}{2ct}\right)^k\left(\frac{1}{2}-\frac{ ||\underline{\bf x}_3||}{2ct}\right)^{a_1+b_1-3-k}-\left(\frac{1}{2}-\frac{ ||\underline{\bf x}_3||}{2ct}\right)^i\left(\frac{1}{2}+\frac{ ||\underline{\bf x}_3||}{2ct}\right)^{a_1+b_1-3-k}\right]\\
&=\frac{1}{2^{a_1+b_1}\pi} \frac{1}{(ct)^{a_1+b_1-1}||\underline{\bf x}_3||}\frac{(a_1+b_1-1)(a_1+b_1-2)}{(a_1-1)(b_1-1)}\sum_{k=a_1-1}^{a_1+b_1-3}\binom{a_1+b_1-3}{k}\notag\\
&\quad\cdot\left[ \left(ct+||\underline{\bf x}_3||\right)^k\left(ct- ||\underline{\bf x}_3||\right)^{a_1+b_1-3-k}-\left(ct- ||\underline{\bf x}_3||\right)^k\left(ct+ ||\underline{\bf x}_3||\right)^{a_1+b_1-3-k}\right],
\end{align*}
where $\underline{\bf x}_3\in\mathcal{B}_{ct}^3$. 
\section{Explicit probability density functions}\label{sec3}

From Theorem \ref{th1} emerges that in order to explicit the density function of $\{\underline{\bf X}_d(t),t>0\},$ we should be able to calculate the following integral
\begin{align}\label{eq:mainint}
&\int_{S_n}f({\bf \underline\tau}_n; \mathbf{\underline{a}}_{n},\mathbf{\underline{b}}_{n},t)\prod_{k=1}^{n+1}\left\{\frac{J_{\frac d2-1}(c\tau_k\rho)}{(c\tau_k\rho)^{\frac d2-1}} \right\}\prod_{k=1}^{n}d\tau_k
\end{align}
appearing in \eqref{eq:density} or equivalently we should calculate the $n$-fold integral appearing in the characteristic function \eqref{eq:cf}.

In general it is not possible to obtain the exact value of \eqref{eq:mainint} (or \eqref{eq:cf}). Nevertheless, for some values of the parametric vectors $\mathbf{\underline{a}}_{n}$ and $\mathbf{\underline{b}}_{n}$ of the generalized Dirichlet disitribution \eqref{eq:gdd}, the integral \eqref{eq:mainint} (or \eqref{eq:cf}) can be worked out. Therefore in this section we study the random walks derived from the suitable choices of the parameters $\mathbf{\underline{a}}_{n}$ and $\mathbf{\underline{b}}_{n}$. In particular we consider two different families of  ``solvable random walks". In our context with the terminology ``solvable random walks", we mean random motions, with a fixed number of steps $n+1$, with isotropic density function of the following type
\begin{equation}\label{eq:introd2}
p_n(\underline{\bf x}_d,t)=A(c^2t^2-||\underline{\bf x}_d||^2)^{b},
\end{equation}
where $b$ is a constant depending on $n$ and $d$, while $A$ is the necessary normalizing factor (depending on $t$). Clearly the solvable random walks represent a sub-family of the general class $\{\underline{\bf X}_d(t),t>0\}$.

We will use in the proof below the same approach developed in De Gregorio and Orsingher (2012).  For this reason we need the following formulae

\begin{equation}\label{eq:int1}
\int_0^ax^\mu(a-x)^{\nu+i} J_\mu(x)J_\nu(a-x)dx=
\frac{\Gamma(\mu+\frac12)\Gamma(\nu+i+\frac12)}{\sqrt{2\pi}\Gamma(\mu+\nu+i+1)}a^{\mu+\nu+i+\frac12}J_{\mu+\nu+\frac12}(a),
\end{equation}
where $i=0,1$,  and $Re\,\mu>-\frac12$, $Re\,\nu>-\frac{i+1}{2}$ (see Gradshteyn and Ryzhik, 1980, page 743, formulae 6.581(3)-(4)), and

 \begin{equation}\label{eq:int2}
\int_0^a\frac{J_\mu(x)J_\nu(a-x)}{x(a-x)^i}dx=\left(\frac1\mu+\frac{i}{\nu}\right)\frac{J_{\mu+\nu}(a)}{a^i},
\end{equation}
where $i=0,1$,  and $Re\,\mu>0$, $Re\,\nu>-(i+1)$ (see Gradshteyn and Ryzhik, 1980, page 678, formulae 6.533(1)-(2)).

\subsection{Solvable processes of first type}\label{subsec:solft}

Let us consider a random motion with random position \eqref{eq:definition}, and fix a time interval $\tau_j,j\geq 1,$ where the process changes the distribution of the displacements. This means that $\underline{\tau}_n$ has density function $GD(\underline{\bf a}_n,\underline{\bf b}_n)$ with parameters $\underline{\bf a}_n=(a_1,...,a_k,...,a_n)$ and $\underline{\bf b}_n=(b_1,...,b_k,...,b_n)$ defined respectively as follows:
\begin{itemize}
\item if $n\leq j$, we assume that $GD(\underline{\bf a}_n,\underline{\bf b}_n)$ possesses parameters equal to
\begin{align}\label{eq:param3}
{\bf a}_n=(d-1,...,d-1),\quad {\bf b}_n=(1,...,1, d-1),
\end{align}
with $d\geq 3$;
\item if $n>j$, one has that 
\begin{align}
&a_k=\begin{cases}\label{eq:param1}
d-1,& k\in\{1,...,j\},\\
\frac d2-1,& k\in\{j+1,...,n\},
\end{cases}\\
&b_k=
\begin{cases}\label{eq:param2}
1,&k\in\{1,...,n-1\}\setminus \{j\},\\
(n-j+1)(\frac d2-1)+i+h+1,&k=j,\\
\frac d2-i,&k=n,
\end{cases}
\end{align}
where $i,h\in\{0,1\}$ and $d\geq 3$.
\end{itemize} 
We indicate this process by $\underline{\bf X}_d^{h,i,j}=\{\underline{\bf X}_d^{h,i,j}(t),t>0\}$ which is a random walk with  ``switching phase". In other words,  we assume that till a fixed step the displacements of the motion follows a standard Dirichlet distribution \eqref{eq:param3}. When the number of deviations is greater than $j$ the random walks changes phase and the assumptions \eqref{eq:param1} and \eqref{eq:param2} fulfills.

The random motions $\underline{\bf X}^{h,i,j}_d$ represent a class of random walks where each combination of the indexes $h,i,j$ defines a different process.  

\begin{theorem}\label{teo:4ident}
Fixed $j\in\mathbb N$, we have that the random motions $\underline{\bf X}_d^{0,0,j}$, $\underline{\bf X}_d^{1,0,j}$, $\underline{\bf X}_d^{0,1,j}$, $\underline{\bf X}_d^{1,1,j}$ are identically distributed.
\end{theorem}
\begin{proof} Let us start by calculating the characteristic function $$\mathcal F_n^{h,i,j}(\underline\alpha_d)={\bf E}\left.\left\{e^{i<\underline{\alpha}_d,\underline{\bf X}^{h,i,j}_d>}\right|\mathcal{N}(t)=n\right\}$$ of $\underline{\bf X}_d^{h,i,j}$. 

Under the assumption  \eqref{eq:param3}, the process $\underline{\bf X}_d^{h,i,j}$ does not depend on $h,i$ and $j$ and reduces to the standard case handled in De Gregorio and Orsingher (2012). Hence we focus our attention on the case $j<n$.
 
Under the assumptions \eqref{eq:param1} and \eqref{eq:param2}, the characteristic function \eqref{eq:cf} becomes
\begin{align*}
&\mathcal F_n^{h,i,j}(\underline\alpha_d)\\
&=\left\{2^{\frac d2-1}\Gamma\left(\frac d2\right)\right\}^{n+1}C(\mathbf{\underline{a}}_{n},\mathbf{\underline{b}}_{n},t)\int_{S_{j-1}^1}\prod_{k=1}^{j-1}\left\{\tau_k^{d-2}\frac{J_{\frac d2-1}(c\tau_k||\underline{\alpha}_d||)}{(c\tau_k||\underline{\alpha}_d||)^{\frac d2-1}}\right\}d\tau_1\cdots d\tau_{j-1}\\
&\quad\cdot\int_{S_j^2}\tau_j^{\frac d2-2}(t-\sum_{k=1}^{j}\tau_k)^{(n-j+1)(\frac d2-1)+i+h}\frac{J_{\frac d2-1}(c\tau_j||\underline{\alpha}_d||)}{(c\tau_j||\underline{\alpha}_d||)^{\frac d2-1}}d\tau_j\\\
&\quad\cdot \int_{S_{n-j}^3}\prod_{k=j+1}^{n}\left\{\tau_k^{\frac d2-2}\frac{J_{\frac d2-1}(c\tau_k||\underline{\alpha}_d||)}{(c\tau_k||\underline{\alpha}_d||)^{\frac d2-1}}\right\}(t-\sum_{k=1}^n\tau_k)^{\frac d2-i-1}\frac{J_{\frac d2-1}(c(t-\sum_{k=1}^n\tau_k)||\underline{\alpha}_d||)}{(c(t-\sum_{k=1}^n\tau_k)||\underline{\alpha}_d||)^{\frac d2-1}}d\tau_{j+1}\cdots d\tau_n
\end{align*}
where 
\begin{align*}
&C(\mathbf{\underline{a}}_{n},\mathbf{\underline{b}}_{n},t)\\
&=\frac{1}{t^{2(n+1)\left(\frac d2-1\right)+h+j}}\frac{\Gamma\left(2(n+1)(\frac d2-1)+h+j+1\right)\Gamma\left((n-j+1)(\frac d2-1)+1-i\right)}{\left(\Gamma(d-1)\right)^j\left(\Gamma(\frac d2-1)\right)^{n-j}\Gamma(\frac d2-i)\Gamma\left(2(n-j+1)(\frac d2-1)+h+1\right)}
\end{align*}
and 
\begin{align}
&S_{j-1}^1=\left\{(\tau_1,...,\tau_{j-1})\in \mathbb{R}^{j-1}:0<\tau_k<t-\sum_{i=0}^{k-1}\tau_i,\,k=1,...,j-1\right\},\label{eq:dom1}\\
&S_j^2=\left\{\tau_j\in \mathbb{R}:0<\tau_j<t-\sum_{i=0}^{j-1}\tau_i\right\},\label{eq:dom2}\\
&S_{n-j}^3=\left\{(\tau_{j+1},...,\tau_{n})\in \mathbb{R}^{n-j}:0<\tau_k<t-\sum_{i=0}^{k-1}\tau_i,\,k=j+1,...,n\right\}.\label{eq:dom3}
\end{align}

The first step consists in the calculation of the $(n-j)-$fold integral
\begin{align*}
&I_1(\underline{\tau}_{j})\\
&=\int_{S_{n-j}^3}\prod_{k=j+1}^{n}\left\{\tau_k^{\frac d2-2}\frac{J_{\frac d2-1}(c\tau_k||\underline{\alpha}_d||)}{(c\tau_k||\underline{\alpha}_d||)^{\frac d2-1}}\right\}(t-\sum_{k=1}^n\tau_k)^{\frac d2-i-1}\frac{J_{\frac d2-1}(c(t-\sum_{k=1}^n\tau_k)||\underline{\alpha}_d||)}{(c(t-\sum_{k=1}^n\tau_k)||\underline{\alpha}_d||)^{\frac d2-1}}d\tau_{j+1}\cdots d\tau_n.
\end{align*}
We apply recursively the result \eqref{eq:int2}. Indeed, the first integral with respect to $\tau_n$ becomes
\begin{align*}
&\frac{1}{(c||\underline{\alpha}_d||)^{d-3-i}}\int_0^{t-\sum_{k=1}^{n-1}\tau_k}\frac{J_{\frac d2-1}(c\tau_n||\underline{\alpha}_d||)}{c\tau_n||\underline{\alpha}_d||}\frac{J_{\frac d2-1}(c(t-\sum_{k=1}^n\tau_k)||\underline{\alpha}_d||)}{(c(t-\sum_{k=1}^n\tau_k)||\underline{\alpha}_d||)^i}d\tau_n\\
&=(y=c\tau_n||\underline{\alpha}_d||)\\
&=\frac{1}{(c||\underline{\alpha}_d||)^{d-2-i}}\int_0^{c(t-\sum_{k=1}^{n-1}\tau_k)||\underline{\alpha}_d||}\frac{J_{\frac d2-1}(y)}{y}\frac{J_{\frac d2-1}(c(t-\sum_{k=1}^{n-1}\tau_k)||\underline{\alpha}_d||-y)}{(c(t-\sum_{k=1}^{n-1}\tau_k)||\underline{\alpha}_d||-y)^i}dy\\
&=\frac{1}{(c||\underline{\alpha}_d||)^{d-2-i}}\frac{1+i}{\frac d2-1}\frac{J_{2\left(\frac d2-1\right)}(c(t-\sum_{k=1}^{n-1}\tau_k)||\underline{\alpha}_d||)}{(c(t-\sum_{k=1}^{n-1}\tau_k)||\underline{\alpha}_d||)^i}.
\end{align*}
The second integral with respect to $\tau_{n-1}$ in $I_1(\underline{\tau}_{j})$ reads
\begin{align*}
&\frac{1}{(c||\underline{\alpha}_d||)^{\frac32 d-4-i}}\frac{1+i}{\frac d2-1}\int_0^{t-\sum_{k=1}^{n-2}\tau_k}\frac{J_{\frac d2-1}(c\tau_{n-1}||\underline{\alpha}_d||)}{c\tau_{n-1}||\underline{\alpha}_d||}\frac{J_{2\left(\frac d2-1\right)}(c(t-\sum_{k=1}^{n-1}\tau_k)||\underline{\alpha}_d||)}{(c(t-\sum_{k=1}^{n-1}\tau_k)||\underline{\alpha}_d||)^i}d\tau_{n-1}\\
&=(y=c\tau_n||\underline{\alpha}_d||)\\
&=\frac{1}{(c||\underline{\alpha}_d||)^{\frac32 d-3-i}}\frac{1+i}{\frac d2-1}\int_0^{c(t-\sum_{k=1}^{n-2}\tau_k)||\underline{\alpha}_d||}\frac{J_{\frac d2-1}(y)}{y}\frac{J_{2\left(\frac d2-1\right)}(c(t-\sum_{k=1}^{n-2}\tau_k)||\underline{\alpha}_d||-y)}{(c(t-\sum_{k=1}^{n-2}\tau_k)||\underline{\alpha}_d||-y)^i}dy\\
&=\frac{1}{(c||\underline{\alpha}_d||)^{\frac32 d-3-i}}\frac{(1+i)(2+i)}{2(\frac d2-1)^2}\frac{J_{3\left(\frac d2-1\right)}(c(t-\sum_{k=1}^{n-2}\tau_k)||\underline{\alpha}_d||)}{(c(t-\sum_{k=1}^{n-2}\tau_k)||\underline{\alpha}_d||)^i}.
\end{align*}
 Therefore, by continuing at the same way with the successive integrations we obtain that
$$I_1(\underline{\tau}_j)=\frac{1}{(c||\underline{\alpha}_d||)^{(n-j+1)(\frac d2-1)-i}}\frac{(n-j+i)!/(n-j)!}{(\frac d2-1)^{n-j}}\frac{J_{(n-j+1)\left(\frac d2-1\right)}(c(t-\sum_{k=1}^{j}\tau_k)||\underline{\alpha}_d||)}{(c(t-\sum_{k=1}^{j}\tau_k)||\underline{\alpha}_d||)^i}.$$
Now we work out the following integral 
$$I_2(\underline{\tau}_{j-1})
=\int_{S_j^2}\tau_j^{d-2}(t-\sum_{k=1}^{j}\tau_k)^{(n-j+1)(\frac d2-1)+i+h}\frac{J_{\frac d2-1}(c\tau_j||\underline{\alpha}_d||)}{(c\tau_j||\underline{\alpha}_d||)^{\frac d2-1}}I_1(\underline{\tau}_j)d\tau_j$$
by applying \eqref{eq:int1}. Therefore, we get that
\begin{align*}
I_2(\underline{\tau}_{j-1})
&=\frac{1}{(c||\underline{\alpha}_d||)^{(2(n-j+1)+2)(\frac d2-1)+h+1}}\frac{(n-j+i)!/(n-j)!}{(\frac d2-1)^{n-j}}\int_0^{c(t-\sum_{k=1}^{j-1})||\underline{\alpha}_d||}y^{\frac d2-1}J_{\frac d2-1}(y)\\
&\quad\cdot(c(t-\sum_{k=1}^{j-1}\tau_k)||\underline{\alpha}_d||-y)^{(n-j+1)\left(\frac d2-1\right)+h}J_{(n-j+1)\left(\frac d2-1\right)}(c(t-\sum_{k=1}^{j-1}\tau_k)||\underline{\alpha}_d||-y)dy\\
&=\frac{\frac{(n-j+i)!/(n-j)!}{(\frac d2-1)^{n-j}}}{(c||\underline{\alpha}_d||)^{(2(n-j+1)+2)(\frac d2-1)+h+1}}\frac{\Gamma(\frac{d-1}{2})\Gamma((n-j+1)(\frac d2-1)+h+\frac12)}{\sqrt{2\pi}\Gamma((n-j+2)(\frac d2-1)+h+1)}\\
&\quad\cdot(c(t-\sum_{k=1}^{j-1}\tau_k)||\underline{\alpha}_d||)^{(n-j+2)\left(\frac d2-1\right)+h+\frac12}J_{(n-j+2)\left(\frac d2-1\right)+\frac12}(c(t-\sum_{k=1}^{j-1}\tau_k)||\underline{\alpha}_d||).
\end{align*}

The last step consists in the calculation of 

\begin{equation}\label{eq:laststep}
\int_{S_{j-1}^1}\prod_{k=1}^{j-1}\left\{\tau_k^{d-2}\frac{J_{\frac d2-1}(c\tau_k||\underline{\alpha}_d||)}{(c\tau_k||\underline{\alpha}_d||)^{\frac d2-1}}\right\}I_2(\underline{\tau}_{j-1})d\tau_1\cdots d\tau_{j-1}.
\end{equation}
The first integral in \eqref{eq:laststep} becomes
\begin{align*}
&\frac{\frac{(n-j+i)!/(n-j)!}{(\frac d2-1)^{n-j}}}{(c||\underline{\alpha}_d||)^{(2(n-j+1)+4)(\frac d2-1)+h+1}}\frac{\Gamma(\frac{d-1}{2})\Gamma((n-j+1)(\frac d2-1)+h+\frac12)}{\sqrt{2\pi}\Gamma((n-j+2)(\frac d2-1)+h+1)}\\
&\quad\cdot\int_0^{t-\sum_{k=1}^{j-2}\tau_k}(c\tau_k||\underline{\alpha}_d||)^{\frac d2-1}J_{\frac d2-1}(c\tau_k||\underline{\alpha}_d||)\\
&\quad\cdot(c(t-\sum_{k=1}^{j-1}\tau_k)||\underline{\alpha}_d||)^{(n-j+2)\left(\frac d2-1\right)+h+\frac12}J_{(n-j+2)\left(\frac d2-1\right)+\frac12}(c(t-\sum_{k=1}^{j-1}\tau_k)||\underline{\alpha}_d||)d\tau_{j-1}\\
&=\frac{\frac{(n-j+i)!/(n-j)!}{(\frac d2-1)^{n-j}}}{(c||\underline{\alpha}_d||)^{(2(n-j+1)+4)(\frac d2-1)+h+2}}\frac{\left(\Gamma(\frac{d-1}{2})\right)^2\Gamma((n-j+1)(\frac d2-1)+h+\frac12)}{(\sqrt{2\pi})^2\Gamma((n-j+3)(\frac d2-1)+h+\frac32)}\\
&\quad\cdot(c(t-\sum_{k=1}^{j-2}\tau_k)||\underline{\alpha}_d||)^{(n-j+3)\left(\frac d2-1\right)+h+1}J_{(n-j+3)\left(\frac d2-1\right)+1}(c(t-\sum_{k=1}^{j-2}\tau_k)||\underline{\alpha}_d||),
\end{align*}
where in the last step we have applied \eqref{eq:int1}. Therefore, by continuing at the same way, we get that
\begin{align*}
&\int_{S_{j-1}^1}\prod_{k=1}^{j-1}\left\{\tau_k^{d-2}\frac{J_{\frac d2-1}(c\tau_k||\underline{\alpha}_d||)}{(c\tau_k||\underline{\alpha}_d||)^{\frac d2-1}}\right\}I_2(\underline{\tau}_{j-1})d\tau_1\cdots d\tau_{j-1}\\
&=\frac{\frac{(n-j+i)!/(n-j)}{(\frac d2-1)^{n-j}}}{(c||\underline{\alpha}_d||)^{2(n+1)(\frac d2-1)+h+j}}\frac{\left(\Gamma(\frac{d-1}{2})\right)^j\Gamma((n-j+1)(\frac d2-1)+h+\frac12)}{(\sqrt{2\pi})^j\Gamma((n+1)(\frac d2-1)+h+\frac{j+1}{2})}\\
&\quad\cdot(ct||\underline{\alpha}_d||)^{(n+1)\left(\frac d2-1\right)+h+\frac j2}J_{(n+1)\left(\frac d2-1\right)+\frac j2}(ct||\underline{\alpha}_d||)
\end{align*}

Finally, the characteristic function of $\underline{\bf X}_d^{h,i,j}$ is equal to
\begin{align*}
&\mathcal F_n^{h,i,j}(\underline\alpha_d)\\
&=\left\{2^{\frac d2-1}\Gamma\left(\frac d2\right)\right\}^{n+1}\frac{(n-j+i)!/(n-j)!}{(\frac d2-1)^{n-j}}\frac{\left(\Gamma(\frac{d-1}{2})\right)^j\Gamma((n-j+1)(\frac d2-1)+h+\frac12)}{(\sqrt{2\pi})^j\Gamma((n+1)(\frac d2-1)+h+\frac{j+1}{2})}\\
&\quad\cdot\frac{\Gamma\left(2(n+1)(\frac d2-1)+h+j+1\right)\Gamma\left((n-j+1)(\frac d2-1)+1-i\right)}{\left(\Gamma(d-1)\right)^j\left(\Gamma(\frac d2-1)\right)^{n-j}\Gamma(\frac d2-i)\Gamma\left(2(n-j+1)(\frac d2-1)+h+1\right)}\frac{J_{(n+1)\left(\frac d2-1\right)+\frac j2}(ct||\underline{\alpha}_d||)}{(ct||\underline{\alpha}_d||)^{(n+1)(\frac d2-1)+\frac j2}}.
\end{align*}
For $h=1$,
by applying the duplication formula for Gamma functions, we have that
$$\Gamma\left(\frac d2\right)=\sqrt{\pi}2^{2-d}\frac{\Gamma(d-1)}{\Gamma(\frac{d-1}{2})}$$
$$\Gamma\left((n+1)\left(\frac d2-1\right)+\frac{j+1}{2}\right)=\sqrt{\pi}2^{-2(n+1)(\frac d2-1)-j}\frac{\Gamma(2(n+1)(\frac d2-1)+1)}{\Gamma\left((n+1)(\frac d2-1)+\frac j2+1\right)}$$
$$\frac{\sqrt{\pi}\Gamma\left(2(n-j+1)\left(\frac d2-1\right)+2\right)}{\Gamma\left((n-j+1)\left(\frac d2-1\right)+\frac32\right)}=2^{2(n-j+1)\left(\frac d2-1\right)+1}\Gamma\left((n-j+1)\left(\frac d2-1\right)+1\right)$$
and after careful calculations we obtain that
\begin{align}\label{eq:cfgeneral}
&\mathcal F_n^{1,i,j}(\underline\alpha_d)=2^{(n+1)(\frac d2-1)+\frac j2}\Gamma\left((n+1)\left(\frac d2-1\right)+\frac j2+1\right)\frac{J_{(n+1)\left(\frac d2-1\right)+\frac j2}(ct||\underline{\alpha}_d||)}{(ct||\underline{\alpha}_d||)^{(n+1)(\frac d2-1)+\frac j2}}\end{align}
which does not depend on $i$.
For $h=0$ we take into account the following relationships 
$$\Gamma\left((n-j+1)\left(\frac d2-1\right)+\frac{1}{2}\right)=\sqrt{\pi}2^{-2(n-j+1)(\frac d2-1)}\frac{\Gamma(2(n-j+1)(\frac d2-1)+1)}{\Gamma\left((n-j+1)(\frac d2-1)+1\right)}$$
$$\frac{\sqrt{\pi}\Gamma\left(2(n+1)\left(\frac d2-1\right)+j+1\right)}{\Gamma\left((n+1)\left(\frac d2-1\right)+\frac j2+1\right)}=2^{2(n+1)\left(\frac d2-1\right)+j}\Gamma\left((n+1)\left(\frac d2-1\right)+\frac{j+1}{2}\right)$$
which leads to \eqref{eq:cfgeneral}. 

Fixed $j\in\{1,...,n-1\}$, the processes $\underline{\bf X}_d^{0,0,j}$, $\underline{\bf X}_d^{1,0,j}$, $\underline{\bf X}_d^{0,1,j}$, $\underline{\bf X}_d^{1,1,j}$ have the same characteristic function \eqref{eq:cfgeneral} and then the same probability distribution.

\end{proof}

In view of Theorem \ref{teo:4ident}, we point out that  the conditional probability distribution of $\underline{\bf X}_d^{h,i,j}$, for $j<n$, does not depend on the choice of the indexes $h,i$. Now, we provide the main result of this section.
\begin{theorem}\label{teo:denssw}
For $d\geq 3$, and $j\in\{1,...,n-1\}$, the random processes $\underline{\bf X}_d^{h,i,j}$ have
 density function given by 
\begin{align}\label{eq:densxhij}
p_n^j(\underline{\bf x}_d,t)&=\frac{P_n(\underline{{\bf X}}_d^{h,i,j}\in d\underline{{\bf x}}_d)}{\prod_{i=1}^ddx_i}\notag\\
&=
\frac{\Gamma\left((n+1)\left(\frac d2-1\right)+\frac j2+1\right)}{\Gamma(n(\frac d2-1)+\frac {j}{2})}\frac{(c^2t^2- ||\underline{\bf x}_d||^2)^{n(\frac d2-1)+\frac {j}{2}-1}}{\pi^{d/2}(ct)^{2(n+1)(\frac d2-1)+j}},
\end{align}
with $\underline{\bf x}_d\in \mathcal B_{ct}^d$.
\end{theorem}
\begin{proof}
We get the density function of $\underline{\bf X}_d^{h,i,j}$ by inverting the characteristic function \eqref{eq:cfgeneral}. For $j\in\{1,...,n-1\}$, by following the same steps of the proof of Theorem 2 in De Gregorio and Orsingher (2012), we get the result \eqref{eq:densxhij}.
\end{proof}

\begin{remark}
The isotropic Pearson random walks admit density functions expressed in closed form just in few cases (see, for instance, Stadje, 1987, Orsingher and De Gregorio, 2007, Le Ca\"er, 2010 and De Gregorio and Orsingher, 2012). For this reason, the result \eqref{eq:densxhij} is interesting, because it allows to claim that $\underline{\bf X}_d^{h,i,j}$ represents an entire class of random walks which have explicit probability distributions. 
\end{remark}

\begin{remark}
As the reader can check \eqref{eq:densxhij} coincides with \eqref{eq:introd2} with 
$$A=\frac{1}{\pi^{d/2}(ct)^{2(n+1)(\frac d2-1)+j}}\frac{\Gamma\left((n+1)\left(\frac d2-1\right)+\frac j2+1\right)}{\Gamma(n(\frac d2-1)+\frac {j}{2})},\quad
 b=
 n\left(\frac d2-1\right)+\frac {j}{2}-1.
 $$ 
 Furthermore, for $j\geq n$, the random walk $\underline{\bf X}_d^{h,i,j}$ has density function given by (2.10) in De Gregorio and Orsingher (2012). 
\end{remark}


Bearing in mind the considerations done in Remark \ref{remis} about the isotropy of the random process, for $j<n$, we can write that
\begin{align*}
P_n(\underline{\bf X}_d^{h,i,j}\in \mathcal{B}_z^d)&=P_n(||\underline{\bf X}_d^{h,i,j}||<z)\\
&=\frac{2\pi^{\frac d2}}{\Gamma(\frac d2)}\int_0^z  r^{d-1}p_n^j(r,t)dr\\
&=\frac{\Gamma\left((n+1)\left(\frac d2-1\right)+\frac j2+1\right)}{\Gamma(n(\frac d2-1)+\frac {j}{2})\Gamma(\frac d2)}\int_0^{z^2/c^2t^2}r^{\frac d2-1}(1-r)^{n(\frac d2-1)+\frac {j}{2}-1}dr\\
&=\frac{B(\frac{z^2}{c^2t^2};\frac d2, n(\frac d2-1)+\frac {j}{2})}{B(\frac d2, n(\frac d2-1)+\frac {j}{2})},
\end{align*}
with $z<ct$. Then, if $\frac d2$ and $n(\frac d2-1)+\frac {j}{2}$ assume integer values (happening when both $\frac d2$ and $\frac {j}{2}$ are even), by means of the relationship \eqref{eq:relincombeta}, we get that
\begin{align}\label{eq:cdf}
P_n(\underline{\bf X}_d^{h,i,j}\in \mathcal{B}_z^d)
=\sum_{k=\frac d2}^{ (n+1)(\frac d2-1)+\frac {j}{2}}\binom{(n+1)(\frac d2-1)+\frac {j}{2}}{k}\left(\frac {z^2}{c^2t^2}\right)^k\left(1-\frac {z^2}{c^2t^2}\right)^{(n+1)(\frac d2-1)+\frac {j}{2}-k},
\end{align}
with $z<ct$.

\subsection{Solvable processes of second type}\label{subsec:solst}

 Let us indicate by $\underline{\bf Y}_d^{h,i}=\{\underline{\bf Y}_d^{h,i}(t),t>0\}$ the random walk \eqref{eq:definition} where the vector ${\bf \underline\tau}_n$ has probability distribution \eqref{eq:gdd} with parameters
\begin{equation}\label{eq:paramy1}
\underline{\bf a}_n=\left(\frac dh-1,...,\frac dh-1\right),\
\end{equation}
\begin{equation}\label{eq:paramy2}
 \underline{\bf b}_n=\left(1,...,1,\frac dh-i\right),
 \end{equation} with $h\in\{1,2\}$, $i\in\{0,1\}$ and $d\geq 2$.  These random walks coincides with the random flights with standard Dirichlet steps studied (independently) by De Gregorio and Orsingher (2012) and Le Ca\"er (2010) in the case $i=1$. The case $i=0$ has been treated by Le Ca\"er (2010) (actually the vector of parameters has values in a different order). 

Furthermore, we consider the switching phase random motions (in the sense explained in the previous section) $\underline{\bf Z}_d^{h,j}=\{\underline{\bf Z}_d^{h,j}(t),t>0\}$ defined by \eqref{eq:definition} and with lengths of the interval between consecutive changes of direction possessing generalized Dirichlet distribution \eqref{eq:gdd}. In this case the vector $\underline{\bf a}_n$ is defined as in \eqref{eq:paramy1} while:
\begin{itemize}
\item for $j\in\{1,...,n-1\}$, $h\in\{1,2\}$ and $d\geq 2$, the vector $\underline{\bf b}_n$ has entries given by
\begin{align} 
b_k=\begin{cases}\label{eq:paramz2}
1,&k\in\{1,...,n-1\}\setminus \{j\},\\
2,&k=j,\\
\frac dh-1,&k=n;
\end{cases}
\end{align}
\item for $j\in\{n,n+1,...\}$ and $d\geq 2$, one has that  $\underline{\bf b}_n=\left(1,...,1,\frac dh-1\right)$.

\end{itemize}

\begin{theorem}
Fixed $j\in\mathbb N$ and $h\in\{1,2\},$ the random processes
$\underline{\bf Y}_d^{h,i}$, $i\in\{0,1\}$, and $\underline{\bf Z}_d^{h,j}$, are identically distributed.
\end{theorem}
\begin{proof}
As observed by Le Ca\"er (2010) (see Section 5 of the cited reference) the random motions  $\underline{\bf Y}_d^{h,0}$ and $\underline{\bf Y}_d^{h,1}$ have the same probability distribution and then possess the same characteristic function. The characteristic function of  $\underline{\bf Y}_d^{1,1}$ has been obtained in De Gregorio and Orsingher (2012), formula (2.1), and yields
\begin{equation}\label{eq:cf1}
\hat{\mathcal F}_n^{1,1}(\underline\alpha_d)=\frac{2^{\frac{n+1}{2}(d-1)-\frac12}\Gamma(\frac{n+1}{2}(d-1)+\frac12)}{(ct||\underline{\alpha}_d||)^{\frac{n+1}{2}(d-1)-\frac12}}J_{\frac{n+1}{2}(d-1)-\frac12}(ct||\underline{\alpha}_d||),
\end{equation}
where $d\geq 2$, while the characteristic function of  $\underline{\bf Y}_d^{2,1}$ has been obtained in De Gregorio and Orsingher (2012), formula (2.2), and yields  
\begin{equation}\label{eq:cf2}
\hat{\mathcal F}_n^{2,1}(\underline\alpha_d)=\frac{2^{(n+1)\left(\frac d2-1\right)}\Gamma((n+1)\left(\frac d2-1\right)+1)}{(ct||\underline{\alpha}_d||)^{(n+1)\left(\frac d2-1\right)}}J_{(n+1)\left(\frac d2-1\right)}(ct||\underline{\alpha}_d||),
\end{equation}
where $d\geq 3$.
In order to complete the proof we show that, for $j<n$ (the case $j\geq n$ is trivial), the characteristic function of the random flight $\underline{\bf Z}_d^{h,j}$ coincides with \eqref{eq:cf1} (for $h=1$)  and \eqref{eq:cf2} (for $h=2$). We use similar arguments to those adopted in the proof of Theorem \ref{teo:4ident} and then we omit some details. 

Let us start with the case $h=1$.  Under the assumptions \eqref{eq:paramy1} and \eqref{eq:paramz2}, we have that the characteristic function of $\underline{\bf Z}_d^{1,j}$ becomes
\begin{align}\label{eq:cf3}
\tilde{\mathcal F}_n^{1,j}(\underline\alpha_d)
&=\left\{2^{\frac d2-1}\Gamma\left(\frac d2\right)\right\}^{n+1}\frac{\Gamma((n+1)(d-1)+1)}{(n-j+1)(d-1)(\Gamma(d-1))^{n+1}}\frac{1}{t^{(n+1)(d-1)}}\notag\\
&\quad\cdot\int_{S_{j-1}^1}\prod_{k=1}^{j-1}\left\{\tau_k^{d-2}\frac{J_{\frac d2-1}(c\tau_k||\underline{\alpha}_d||)}{(c\tau_k||\underline{\alpha}_d||)^{\frac d2-1}}\right\}d\tau_1\cdots d\tau_{j-1}\notag\\
&\quad\cdot \int_{S_j^2}\tau_j^{d-2}\frac{J_{\frac d2-1}(c\tau_k||\underline{\alpha}_d||)}{(c\tau_j||\underline{\alpha}_d||)^{\frac d2-1}}(t-\sum_{k=1}^{j}\tau_k)d\tau_j\notag\\
&\quad\cdot \int_{S_{n-j}^3}\prod_{k=j+1}^{n+1}\left\{\tau_k^{d-2}\frac{J_{\frac d2-1}(c\tau_k||\underline{\alpha}_d||)}{(c\tau_k||\underline{\alpha}_d||)^{\frac d2-1}}\right\}d\tau_{j+1}\cdots d\tau_n,
\end{align}
where $S_{j-1}^1,S_j^2$ and $S_{n-j}^3$ are defined by \eqref{eq:dom1}, \eqref{eq:dom2} and \eqref{eq:dom3} respectively.
By applying the result \eqref{eq:int1} for $i=0$ we obtain that the $(n-j)-$fold integral appearing in \eqref{eq:cf3} becomes
\begin{align*}
J_1(\underline{\tau}_j)
&=\int_{S_{n-j}^3}\prod_{k=j+1}^{n+1}\left\{\tau_k^{d-1}\frac{J_{\frac d2-1}(c\tau_k||\underline{\alpha}_d||)}{(c\tau_k||\underline{\alpha}_d||)^{\frac d2-1}}\right\}d\tau_{j+1}\cdots d\tau_n\notag\\
&=\frac{\left(\Gamma(\frac{d-1}{2})\right)^{n-j+1}}{(\sqrt{2\pi})^{n-j}\Gamma(\frac{n-j+1}{2}(d-1))}\frac{(c(t-\sum_{k=1}^{j}\tau_k)||\underline{\alpha}_d||)^{\frac{n-j+1}{2}\left( d-1\right)-\frac12}}{(c||\underline{\alpha}_d||)^{(n-j+1)(d-1)-1}}\\
&\quad\cdot J_{\frac{n-j+1}{2}\left( d-1\right)-\frac12}(c(t-\sum_{k=1}^{j}\tau_k)||\underline{\alpha}_d||)
\end{align*}
and then, by taking into account the result \eqref{eq:int1} for $i=1$, we get that
\begin{align*}
J_2(\underline{\tau}_{j-1})
&=\int_0^{t-\sum_{k=1}^{j-1}\tau_j}\tau_j^{d-2}\frac{J_{\frac d2-1}(c\tau_k||\underline{\alpha}_d||)}{(c\tau_j||\underline{\alpha}_d||)^{\frac d2-1}}(t-\sum_{k=1}^{j}\tau_k)J(\underline{\tau}_j)d\tau_{j}\notag\\
&=\frac{\left(\Gamma(\frac{d-1}{2})\right)^{n-j+2}\frac{n-j+1}{2}(d-1)}{(\sqrt{2\pi})^{n-j+1}\Gamma(\frac{n-j+2}{2}(d-1)+1)}\frac{(c(t-\sum_{k=1}^{j-1}\tau_k)||\underline{\alpha}_d||)^{\frac{n-j+2}{2}\left( d-1\right)+\frac12}}{(c||\underline{\alpha}_d||)^{(n-j+2)(d-1)}}\\
&\quad\cdot J_{\frac{n-j+2}{2}\left( d-1\right)-\frac12}(c(t-\sum_{k=1}^{j-1}\tau_k)||\underline{\alpha}_d||).
\end{align*}
Finally, by using again \eqref{eq:int1} for $i=1$ recursively, the last $(j-1)-$fold integral is equal to
\begin{align}\label{eq:lastexp}
&\int_{S_{j-1}^1}\prod_{k=1}^{j-1}\left\{\tau_k^{d-1}\frac{J_{\frac d2-1}(c\tau_k||\underline{\alpha}_d||)}{(c\tau_k||\underline{\alpha}_d||)^{\frac d2-1}}\right\}J_2(\underline{\tau}_{j-1})d\tau_{1}\cdots d\tau_{j-1}\notag\\
&=\frac{\left(\Gamma(\frac{d-1}{2})\right)^{n+1}\frac{n-j+1}{2}(d-1)}{(\sqrt{2\pi})^{n}\Gamma(\frac{n+1}{2}(d-1)+1)}\frac{(ct||\underline{\alpha}_d||)^{\frac{n+1}{2}\left( d-1\right)+\frac12}}{(c||\underline{\alpha}_d||)^{(n+1)(d-1)}}J_{\frac{n+1}{2}\left( d-1\right)-\frac12}(ct||\underline{\alpha}_d||)
\end{align}
and then by plugging \eqref{eq:lastexp} into \eqref{eq:cf3}, we obtain that $\tilde{\mathcal F}_n^{1,j}(\underline\alpha_d)$ coincides with \eqref{eq:cf1}.

For $h=2$, we can use steps similar to those adopted above together with the result \eqref{eq:int2}. Hence, we conclude that the characteristic function $\tilde{\mathcal F}_n^{2,j}(\underline\alpha_d)$ of $\underline{\bf Z}_d^{2,j}$ is equal to \eqref{eq:cf2}.
\end{proof}

\begin{remark}
The solvable random walks $\underline{\bf Y}_d^{1,i}$ and $\underline{\bf Z}_d^{1,j}$, with $i\in \{1,2\}$, $j\in\mathbb N$ and $d\geq 2$, have the same probability density function
\begin{equation}
\frac{\Gamma(\frac{n+1}{2}(d-1)+\frac12)}{\Gamma(\frac{n}{2}(d-1))}\frac{(c^2t^2- ||\underline{\bf x}_d||^2)^{\frac{n}{2}(d-1)-1}}{\pi^{d/2}(ct)^{(n+1)(d-1)-1}},\quad\underline{\bf x}_d\in\mathcal{B}_{ct}^d,
\end{equation}
 which is obtained by inverting \eqref{eq:cf1} (see (2.10) in De Gregorio and Orsingher, 2012). 

At the same way we have that $\underline{\bf Y}_d^{2,i}$ and $\underline{\bf Z}_d^{2,j}$, with $i\in \{1,2\}$, $j\in\mathbb N$ and $d\geq 3$, have the same probability density function
\begin{equation}
\frac{\Gamma((n+1)(\frac d2-1)+1)}{\Gamma(n(\frac d2-1))}\frac{(c^2t^2- ||\underline{\bf x}_d||^2)^{n(\frac d2-1)-1}}{\pi^{d/2}(ct)^{2(n+1)(\frac d2-1)}},\quad\underline{\bf x}_d\in\mathcal{B}_{ct}^d,
\end{equation}
which is obtained by inverting \eqref{eq:cf2} (see (2.11) in De Gregorio and Orsingher, 2012).

\end{remark}


\section{Some remarks on unconditional probability distributions}\label{sec4}
 So far, we have analyzed the random process \eqref{eq:definition}, with a fixed number of deviations. Formally, we can write the unconditional density function of \eqref{eq:definition} by means of the probability distribution $P(\mathcal N(t)=n)$ of $\mathcal N(t)$, with $n\geq 0$. We recall that $\mathcal N(t)$ is independent from $\underline{\theta}_{d-1}$ and $\underline{\tau}_{n}$. 

The absolutely component of the probability distribution of  $\{\underline{\bf X}_d(t),t>0\}$ is given by
\begin{equation}\label{eq:generaluncondf}
p(\underline{\bf x}_d,t)=\frac{P(\underline{\bf X}_d(t)\in d\underline{\bf x}_d)}{\prod_{k=1}^d dx_k}=\sum_{n=1}^\infty p_n(\underline{\bf x}_d,t)P(\mathcal N(t)=n),
\end{equation}
where $ p_n(\underline{\bf x}_d,t)$ is equals to \eqref{eq:density} and $\underline{\bf x}_d\in  \mathcal{B}_{ct}^d$. Furthermore the distribution of the process $\{\underline{\bf X}_d(t),t>0\}$ has a discrete component given by the probability that the random motion at time $t>0$ hits the edge $\partial  \mathcal{B}_{ct}^d=\mathbb{S}_{ct}^{d-1}=\{\underline{\bf x}_d\in \mathbb{R}^d:||\underline{\bf x}_d||=ct\}$. This probability emerges when the random walk does not change the initial direction, i.e. $\mathcal{N}(t)=0$. Hence, we can write 
$$P(\underline{\bf X}_d(t)\in \partial\mathcal{B}_{ct}^d)=P(\mathcal N(t)=0).$$
Therefore, the random process $\{\underline{\bf X}_d(t),t>0\}$ has support in the closed ball $\overline{\mathcal{B}_{ct}^d}=\mathcal{B}_{ct}^d\cup \partial \mathcal{B}_{ct}^d$ and the complete density function (in sense of generalized functions) reads
\begin{equation}\label{eq:uncondensgf}
p(\underline{\bf x}_d,t){\bf 1}_{\mathcal{B}_{ct}^d}(\underline{\bf x}_d)+P(\mathcal N(t)=0)\delta(ct-||\underline{\bf x}_d||),
\end{equation}
where $\delta(\cdot)$ is the Dirac's delta function.

 Under a suitable assumption on the law of the random number of changes of direction $\mathcal N(t)$, we are able to explicit the probability distribution (or equivalently \eqref{eq:uncondensgf}) for  $\underline{\bf Y}_d^{h,i}$ or equivalentely $\underline{\bf Z}_d^{h,j}$. This problem has been tackled in De Gregorio and Orsingher (2012).

Now we focus our attention on $\underline{\bf Z}_d^{2,j}$, for $\underline{\bf Z}_d^{1,j}$ holds the simlar considerations below. In order to obtain the unconditional probability distributions of $\underline{\bf Z}_d^{2,j}$ De Gregorio and Orsingher (2012) introduced the fractional Poisson process in the spirit of the paper by Beghin and Orsingher (2009). Let $\{\mathcal{N}_d(t),t>0\}$ be the fractional Poisson process  with distribution
\begin{equation}\label{eq:fracpoisson2}
P(\mathcal{N}_d(t)=n)=\frac{1}{E_{\frac d2-1,\frac{d}{2}}(\lambda t)}\frac{(\lambda t)^n}{\Gamma\left(n\left(\frac d2-1\right)+\frac d2\right)},\quad d\geq 3, n\geq 0,
\end{equation}
where $E_{\alpha,\beta}(x)=\sum_{k=0}^\infty \frac{x^k}{\Gamma(\alpha k+\beta)},\, x\in\mathbb{R},\alpha,\beta>0,$ is the generalized Mittag-Leffler function. For $\alpha=\beta=1$ the Mittag-Leffler function $E_{1,1}(x)$ becomes the exponential function and then in this case $\{\mathcal{N}_d(t),t>0\}$ reduces to the standard Poisson process.

If we assume that the number of deviations is a fractional Poisson process with distribution \eqref{eq:fracpoisson2}, the absolutely continuous component of the unconditional probability distributions of $\underline{\bf Z}_d^{2,j}, j\geq 0, d\geq 2,$ is equal to:
\begin{align}\label{eq:uncondlaw}
p(\underline{\bf x}_d,t)&=\frac{P(\underline{\bf Z}_d^{2,j}\in d\underline{\bf x}_d)}{\prod_{k=1}^d dx_k}=\frac{\lambda t(c^2t^2- ||\underline{\bf x}_d||^2)^{\frac{d}{2}-2}}{\pi^{d/2}(ct)^{4(\frac d2-1)}}\frac{E_{\frac d2-1,\frac{d}{2}-1}\left(\frac{\lambda t(c^2t^2- ||\underline{\bf x}_d||^2)^{\frac d2-1}}{(ct)^{d-2}}\right)}{E_{\frac d2-1,\frac{d}{2}}(\lambda t)},
\end{align}
with $\underline{\bf x}_d\in\mathcal{B}_{ct}^d$.

For the above considerations, we have that the probability distribution of $\underline{\bf Z}_d^{2,j}$ admits the following discrete component 
\begin{equation}
P(\underline{\bf Z}_d^{2,j}\in \partial \mathcal{B}_{ct}^d)=\frac{1}{E_{\frac d2-1,\frac{d}{2}}(\lambda t)}\frac{1}{\Gamma\left(\frac{d}{2}\right)}.
\end{equation}
From \eqref{eq:uncondensgf}, it is immediate to observe that the distance processes $R_d^{2,j}=\{R_d^{2,j}(t),t>0\}$, where $R_d^{2,j}(t)=||\underline{\bf Z}_d^{2,j}||$, have probability density function (in the sense of generalized functions) given by

\begin{equation}
r^{d-1}\text{meas}(\mathbb{S}_{1 }^{d-1})p(r,t){\bf 1}_{(0,ct)}(r)+\frac{1}{E_{\frac d2-1,\frac{d}{2}}(\lambda t)}\frac{1}{\Gamma\left(\frac{d}{2}\right)}\delta(ct-r).
\end{equation}

\begin{theorem}
For $p\geq 1$, the $p$-th moment of $R_d^{2,j},d\geq 3,$ becomes
\begin{align}
{\bf E}(R_d^{2,j})^p=\frac{(ct)^p}{E_{\frac d2-1,\frac{d}{2}}(\lambda t)}\left\{\frac{\Gamma(\frac{p+d}{2})}{\Gamma(\frac{d}{2})}\lambda tE_{\frac d2-1,d+\frac{p}{2}-1}(\lambda t)+\frac{1}{\Gamma\left(\frac{d}{2}\right)}\right\}.
\end{align}
\end{theorem}
\begin{proof}
We have to calculate
\begin{align*}
{\bf E}(R_d^{2,j})^p=\int_0^{ct}r^{p+d-1}\text{meas}(\mathbb{S}_{1 }^{d-1})p(r,t)dr+\frac{(ct)^p}{E_{\frac d2-1,\frac{d}{2}}(\lambda t)}\frac{1}{\Gamma\left(\frac{d}{2}\right)}.
\end{align*}
Now, we focus our attention on the integral appearing in the above equality. One has that
\begin{align*}
&\int_0^{ct}r^{p+d-1}\text{meas}(\mathbb{S}_{1 }^{d-1})p(r,t)dr\\
&=\frac{2\lambda t}{\Gamma(\frac d2)(ct)^{4(\frac d2-1)}E_{\frac d2-1,\frac{d}{2}}(\lambda t)}\int_0^{ct}r^{p+d-1}(c^2t^2- r^2)^{\frac{d}{2}-2}E_{\frac d2-1,\frac{d}{2}-1}\left(\frac{\lambda t(c^2t^2- r^2)^{\frac d2-1}}{(ct)^{d-2}}\right)dr\\
&=\frac{2\lambda t}{\Gamma(\frac d2)(ct)^{4(\frac d2-1)}E_{\frac d2-1,\frac{d}{2}}(\lambda t)}\sum_{k=0}^\infty\frac{\frac{(\lambda t)^k}{(ct)^{k(d-2)}}}{\Gamma(k(\frac d2-1)+\frac{d}{2}-1)}\int_0^{ct}r^{p+d-1}(c^2t^2- r^2)^{\frac{d}{2}-2+k(\frac{d}{2}-1)}dr\\
&=\frac{\lambda t (ct)^p}{\Gamma(\frac d2)E_{\frac d2-1,\frac{d}{2}}(\lambda t)}\sum_{k=0}^\infty\frac{(\lambda t)^k}{\Gamma(k(\frac d2-1)+\frac{d}{2}-1)}\int_0^1x^{\frac{p+d}{2}-1}(1-x)^{\frac{d}{2}-2+k(\frac{d}{2}-1)}dx\\
&=\frac{\lambda t (ct)^p\Gamma(\frac{p+d}{2})}{\Gamma(\frac d2)E_{\frac d2-1,\frac{d}{2}}(\lambda t)}\sum_{k=0}^\infty\frac{(\lambda t)^k}{\Gamma(k(\frac d2-1)+d+\frac{p}{2}-1)}
\end{align*}
which concludes the proof.
\end{proof}

\end{document}